\documentclass[journal]{IEEEtran}  

\usepackage{amsmath,amssymb,amsfonts,amsthm}

\usepackage{cite}
\usepackage[pdftex]{graphicx}
\usepackage[hyphens]{url}
\usepackage[hidelinks]{hyperref}
\usepackage{algorithm}
\usepackage{algpseudocode}
\usepackage{textcomp}
\usepackage[caption=false,font=footnotesize]{subfig}
\usepackage{stfloats}
\usepackage{tikz}
\usetikzlibrary{math, positioning, shapes, decorations.pathreplacing, arrows, calc, tikzmark}
\usepackage{multirow}
\usepackage{blkarray}
\usepackage{bigstrut}
\newcommand{\EX}{\mathbb{E}}

\newtheorem{theorem}{Theorem}

\newtheorem{claim}{Claim}

\theoremstyle{remark}
\newtheorem{remark}{Remark}

\newcommand{\AC}[1]{{ #1}}

\title{\LARGE \bf
Dynamic Resource Allocation to Minimize Concave Costs of Shortfalls
(Extended Version)
}
\def\BibTeX{{\rm B\kern-.05em{\sc i\kern-.025em b}\kern-.08em
    T\kern-.1667em\lower.7ex\hbox{E}\kern-.125emX}}
\author{Akhil Bhimaraju, Avhishek Chatterjee, and Lav R. Varshney, \IEEEmembership{Senior Member, IEEE}
\thanks{A. Bhimaraju and L. R. Varshney are with the University
of Illinois Urbana-Champaign (\{akhilb3,varshney\}@illinois.edu). A. Chatterjee is with the Indian Institute of Technology Madras (avhishek@ee.iitm.ac.in).}
\thanks{This work was supported by NSF grant ECCS-2033900, and the grants SERB/SRG/2019/001809 and DST/INSPIRE/04/2016/001171.}
}

\newcommand{\AlgName}{\ensuremath{\textsc{LinAlloc}}}
\newcommand{\NoFbAlgName}{\ensuremath{\textsc{SymAlloc}}}

\newcommand{\NewAfterReview}[1]{{ #1}}
\newcommand{\FinalVersion}[1]{{\color{black} #1}}

\begin{document}

\maketitle
\thispagestyle{empty}
\pagestyle{empty}

\begin{abstract}

We study a resource allocation problem over time, where a finite (random) resource needs to be
distributed among a set of users at each time instant.
Shortfalls in the resource allocated result in user dissatisfaction, 
which we model as an increasing  function of the long-term average shortfall for each user.
In many scenarios such as
wireless multimedia streaming, renewable energy grid, or supply chain logistics, a natural choice for
this cost function turns out to be concave, rather than  usual convex cost functions.
We consider minimizing the (normalized) cumulative cost across  users.
Depending on whether users' mean consumption rates are known or unknown,
this problem can be reduced to two different structured non-convex problems.
The ``known'' case is a concave minimization problem subject to a linear constraint.
By exploiting a well-chosen linearization of the cost functions,  we solve this 
provably within $\mathcal{O}\left(\frac{1}{m}\right)$ of the optimum, in $\mathcal{O}\left(m \log{m}\right)$ time,
where $m$ is the number of users in the system. 
In the ``unknown'' case, 
we are faced with minimizing the sum of functions that are concave on part of the domain
and convex on the rest, subject to a linear constraint.
We present a provably exact 
algorithm 
when the cost functions and prior distributions on mean consumption are the same
across all users. 
\end{abstract}
\begin{IEEEkeywords}
    Optimization, Smart grid, Stochastic systems
\end{IEEEkeywords}

\section{Introduction}
\label{sec:intro}
\IEEEPARstart{N}{on-convex} optimization has seen rapidly increasing  practical interest over the last decade, primarily due to applications in machine learning and related areas \cite{JainK2017}. Theoretical interest in non-convex optimization has also seen steady growth \cite{DauphinPGCGB2014, AnandkumarG2016}.
Theoretical results in non-convex optimization, however, are almost always limited to local optima \cite{DanilovaDGGGKS2022}.

Here, we consider two structured non-convex optimization problems and present polynomial-time algorithms with provable guarantees on reaching their respective {\em global} optima.  These non-convex optimization problems arise in dynamic resource allocation over time, where the goal is to minimize dissatisfaction due to shortfalls in allocated resources.  

In many engineering and business settings, the main challenge is to share  a limited,  time-varying, and random resource with multiple users over time.
Each of these users stores their allocated share in local storage and consumes it over time as needed.
Shortfalls---when there is not enough resource to meet the requirement at that time---cause the user dissatisfaction.
The goal of the resource server is to minimize the long-term dissatisfaction across all users. 
We highlight a few applications where this resource allocation problem is crucial.

{\bf Multimedia streaming. } In wireless multimedia streaming, users receive \NewAfterReview{their desired} content from the access point and store it in their buffers,
which the media players access to play.
Here, the time-varying resource is the amount of data  transmitted over the wireless system, which is random and time-varying due to multipath fading \cite{HouH2015,SinghK2015,BodasSYS2014}. See \cite{BhimarajuZC2022} for a detailed discussion.

{\bf Renewable energy grid. } The energy generated by wind and solar resources
varies over time due to weather and other factors \cite{NeelyTD2010}.
Users store the received energy in local storage \cite{MohamadTL2021},
and the consumption is also random and time-varying \cite{HillSCGG2012}.
Whenever there is not enough energy to meet requirements at that time, the user becomes discontent.
For  domestic users, this may lead to dissatisfaction or annoyance,%
\footnote{\NewAfterReview{With the proliferation of plug-in electric vehicles, domestic users
increasingly have access to a storage system for storing energy for later use.}}
whereas for industrial users this may \NewAfterReview{necessitate expensive backup options for heavy machines with steep financial costs}.

{\bf Supply chain logistics. }
Produced goods are sent to distribution centers before they can finally be sold to end consumers.
However, production can depend on the economy, climate, or upstream manufacturing, and
the amount of produced goods varies over time \cite{ChenSP2013}.
Further, the demand at any time instant at a specific distribution
center is random \cite{ClayG1997,GuptaM2003}.
If a certain distribution center runs out of inventory, it leads to dissatisfaction among the
users it serves.

As we see in Sec.~\ref{sec:model}, the appropriate function that maps the shortfall to dissatisfaction 
in all these applications turns out to be concave.

\subsection{Contributions} 
First, faithful to the listed applications, we formulate a general version of the dynamic resource allocation problem.
We then obtain lower bounds on the long-term average dissatisfaction in terms of two static non-convex optimization problems, depending on whether the mean consumption is known or unknown.
We also prove solving these non-convex problems yields simple optimal policies. 
Next, for known mean consumption, we solve the non-convex optimization problem almost accurately in (almost) linear time
by exploiting a suitable linearization and the underlying combinatorial structure.
Finally, for unknown mean consumption,  under a general assumption on the prior,
we show the corresponding static optimization problem has a unique non-convex structure,
and solve it accurately assuming symmetry across users. 

The rest of this paper is organized as follows.
Sec.~\ref{sec:model} formulates the resource allocation problem.
Sec.~\ref{sec:red2Opt} reduces the problem to a static
non-convex optimization problem and gives a lower bound for both cases:
where we do and do not know the mean consumption
rates of users.
Sec.~\ref{sec:lp} considers the case where we know the mean consumption rates
and develops an $\mathcal{O}(m\log m)$ algorithm for solving
the optimization problem that is within $\mathcal{O}\left(\frac{1}{m}\right)$ of the
optimal cost.
Sec.~\ref{sec:unknown} gives an exact algorithm for the case where
the average consumption rates are unknown, when the users are symmetric.
Sec.~\ref{sec:conclusion} finally concludes with a discussion on
open questions for future work.

\section{System Model}
\label{sec:model}

There are $m$ users in the system, and resource allocation happens 
at discrete time instants $t\in\{1,2,\ldots\}$.
At the beginning of each time slot $t$, the server has a total
available resource of $c(t)$ units.
Of these $c(t)$ units, user $i$ is served with
$S_i(t)$ units of resource, which is put into buffer $Q_i$.
At the end of each time slot $t$, user $i$ tries to consume $F_i(t)$ units
of resource from the buffer.
A \emph{shortfall} occurs at user $i$ if it tries to consume more resource
than is available in the buffer.
Let the amount of shortfall be $\kappa_i(t)$,
which is the amount of resource desired by user $i$ at time $t$ unavailable
in the buffer.
So the buffer evolution, service, and shortfall are connected by the following equations.%
\footnote{We use the notation $(x)^+=\max(x,0)$.}
\begin{align}
Q_i(t+1) &= \left(Q_i(t)+S_i(t)-F_i(t)\right)^+, \label{eq:buf-evolve}\\
\kappa_i(t) &= \big(F_i(t)-(Q_i(t)+S_i(t))\big)^+,\label{eq:shortfall-at-t} \\
\sum_{i=1}^m &S_i(t) \le c(t). \label{eq:capacity-at-t}
\end{align}
The average shortfall at user $i$ is given by $\bar{\kappa}_i$:
\begin{align}
\bar{\kappa}_i = \lim_{T\to\infty}\frac{1}{T}\sum_{t=1}^T \kappa_i(t).
\label{eq:avg-shortfall}
\end{align}
\AC{For each user $i$, its \emph{dissatisfaction} due to shortfall is given by $V_i(\cdot):[0,\infty)\mapsto[0,\infty)$, i.e., the dissatisfaction of user $i$ is $V_i(\bar{\kappa}_i)$.}
We first note that any reasonable dissatisfaction function $V_i(\cdot)$
must satisfy $V_i(0)=0$ and be an increasing function.
\AC{However, unlike many cost functions, 
a \emph{concave} and increasing $V_i(\cdot)$ is a natural choice, as explained below.}

\AC{Consider two different scenarios in the case of the energy grid: $\bar{\kappa}_i$ close $0$ and $\bar{\kappa}_i$ sufficiently large. In the first case, 
a small increase in $\bar{\kappa}_i$ can cause significant increase in
dissatisfaction.}
In electricity markets, home users with an almost shortfall-free service will
experience a greater incremental annoyance for the same amount of increase in shortfall
compared to users who are already experiencing significant shortfall.
Likewise, industrial users will have to budget for \NewAfterReview{backup options for} \AC{heavy machines when $\bar{\kappa}_i$ moves away from $0$.
\NewAfterReview{In contrast}, if $\bar{\kappa}_i$ is sufficiently large, a small decrease in $\bar{\kappa}_i$ would not cause any significant decrease in dissatisfaction.
Thus, in practice, $V_i$ should have a positive but decreasing derivative.
Similar arguments apply to multimedia streaming, as discussed in \cite{BhimarajuZC2022}.}

\subsection{Assumptions}
\label{sec:assumptions}

Our first assumption is that the availability process $c(t)$ has
a well-defined long-term average $\bar{c}$:
\begin{align*}
\bar{c} = \lim_{T\to\infty} \frac{1}{T}\sum_{t=1}^T c(t).
\end{align*}
This holds if the availability process is ergodic and stationary or
cyclo-stationary, which  captures daily, weekly, or seasonal changes in availability.
\NewAfterReview{Note that $c(t)$ can either be an arbitrary deterministic
sequence or a random process, whose long-term average converges to a constant (almost surely).}

We make the following assumptions regarding the consumption processes
$\{F_i(t)\}$.
We assume that for each $i$, $F_i(t)$ has a well-defined long-term
average $f_i$:
\begin{align*}
f_i = \lim_{T\to\infty} \frac{1}{T}\sum_{t=1}^T F_i(t),
\end{align*}
and that $\max_{i,t}F_i(t)$ is bounded  almost surely.

We restrict the set of allowable policies $\{S_i(t)\}$ to those with
a well-defined long-term average $s_i$ for each user $i$:
\begin{align*}
s_i = \lim_{T\to\infty} \frac{1}{T}\sum_{t=1}^T S_i(t).
\end{align*}

\AC{In general, knowledge of $\{Q_i(\tau):\tau\le t\}$ may result in better policies.
Hence, we allow that knowledge while obtaining the lower bound on long-term cost in Theorem~\ref{thm:lower-bound}.
However, in many applications, for example renewable energy and multimedia streaming \cite{BhimarajuZC2022},
$\{Q_i(\tau):\tau\le t\}$ are not available under current operative procedures.
So while designing our policies, we do not use $\{Q_i(\tau):\tau\le t\}$.}

\AC{\subsection{Objective}
Our goal is to develop allocation policy $\{S_i(t)\}$ that minimizes
the long-term average dissatisfaction normalized by the number of users: $\frac{1}{m}\sum_{i=1}^m V_i(\bar{\kappa}_i)$.
Here,  $\bar{\kappa}_i$, the average shortfall for user $i$, depends on the service $\{S_i(t)\}$ and the consumption $\{F_i(t)\}$.

Clearly, the optimum allocation policy depends on the mean consumptions $\{f_i\}$.
When $\{f_i\}$ are known, the objective is to find an allocation policy that minimizes the normalized long-term average dissatisfaction:
\begin{align}
\NewAfterReview{\arg\min_{\{S_i(t)\}}}\frac{1}{m}\sum_{i=1}^m V_i(\bar{\kappa}_i).
\label{eq:avg-dissatisfaction}
\end{align} 

On the other hand, when $\{f_i\}$ are unknown and have prior distributions $\{p_i\}$, a natural objective is an allocation policy that minimizes the {\em expectation} of \eqref{eq:avg-dissatisfaction}: 
\begin{align}
\NewAfterReview{\arg\min_{\{S_i(t)\}}}\frac{1}{m}\sum_{i=1}^m \EX_{f_i \sim p_i} V_i(\bar{\kappa}_i).
\label{eq:EXP-avg-dissatisfaction}
\end{align}}

\begin{remark}
\AC{These objective functions, in particular \eqref{eq:avg-dissatisfaction}, are similar to the objective in \cite{BhimarajuZC2022}. 
However, there are some key differences.
First, unlike the restrictive assumption $V_i(f_i)=V\cdot f_i$  in
\cite{BhimarajuZC2022}, we allow any concave increasing function.
Second, unlike \cite{BhimarajuZC2022},
our availability process $c(t)$ is not assumed to be constant, and our algorithms and their performance guarantees apply to time varying and random $c(t)$.
Third, $F_i(t)$ can take any bounded value here, and we do not assume
$F_i(t)\in\{0,1\}$ as in \cite{BhimarajuZC2022}. 
Fourth, the shortfall studied here, though similar, is not  the same as the ``frequency of pause'' studied in \cite{BhimarajuZC2022}. Finally, the setting that leads to the objective in \eqref{eq:EXP-avg-dissatisfaction}, has not been considered in \cite{BhimarajuZC2022}.}
\label{rem:diff-tcom}
\end{remark}

\section{Reduction to Non-convex Optimization}
\label{sec:red2Opt}

\AC{Our first step towards solving the long-term cost minimization problems in 
\eqref{eq:avg-dissatisfaction} and \eqref{eq:EXP-avg-dissatisfaction} is to reduce them to static non-convex optimization problems,
which if solved, naturally lead to optimal yet simple allocation policies.
We now present a theorem which serves that purpose, and in the following sections 
discuss methods for solving the static non-convex optimization problems efficiently.

\begin{theorem}
When $\{f_i\}$ are known, it is impossible to achieve
a lower long-term average dissatisfaction than
\begin{align}
\bar{V} = \min_{\{s_i\ge0\}}\quad& \frac{1}{m}\sum_{i=1}^mV_i(\max(f_i-s_i,0)) 
\label{eq:conc-prog}\\
\text{subject to}&\ \sum_{i=1}^m s_i \le \bar{c}.\nonumber
\end{align}
When $\{f_i\}$ are unknown, it is impossible to achieve
a lower expected long-term average dissatisfaction than
\begin{align}
\tilde{V} = \min_{\{s_i\ge0\}}\quad& \frac{1}{m}\sum_{i=1}^m\EX_{f_i \sim p_i} V_i(\max(f_i-s_i,0)) \label{eq:exp-conc-prog}\\
\text{subject to}&\ \sum_{i=1}^m s_i \le \bar{c}.\nonumber
\end{align}
Furthermore, if $\{\hat{s}_i\}$ solves \eqref{eq:conc-prog} or \eqref{eq:exp-conc-prog} within an optimality gap $\Delta$, then the simple policy 
$\{S_i(t)=\hat{s}_i \frac{c(t)}{\bar{c}}\}$ solves \eqref{eq:avg-dissatisfaction} or \eqref{eq:EXP-avg-dissatisfaction}, respectively, within an optimality gap $\Delta$.
\label{thm:lower-bound}
\end{theorem}

The main step in Theorem~\ref{thm:lower-bound}'s proof is the following claim relating average shortfall, allocation, and consumption. }

\begin{claim}
Under the assumptions of Sec.~\ref{sec:assumptions}, the average shortfall
at user $i$, $\bar{\kappa}_i$ satisfies
\begin{align*}
\bar{\kappa}_i = \max(f_i-s_i,0),
\end{align*}
where $f_i$ and $s_i$ are the long-term consumption rate
and service rate respectively.
\label{claim:shortfall-equation}
\end{claim}

\begin{proof}
First we consider the case $s_i<f_i$.
Note the following identity.
\begin{align*}
(-x)^+ = (x)^+ - x  \quad \forall \ x \in \mathbb{R}.
\end{align*}
Using $Q_i(t)+S_i(t)-F_i(t)$ for $x$ gives us
\begin{align*}
\kappa_i(t) = Q_i(t+1) - (Q_i(t)+S_i(t)-F_i(t)).
\end{align*}
Adding this equation for $t\in\{1,2,\ldots,T\}$ gives us
\begin{align*}
\sum_{t=1}^T \kappa_i(t) =
Q_i(T+1) - Q_i(1) + \sum_{t=1}^T F_i(t) - \sum_{t=1}^T S_i(t).
\end{align*}
When $s_i<f_i$, \NewAfterReview{the existence of long-term averages
for $\{S_i(t)\}$ and $\{F_i(t)\}$ imply
$\lim_{T\to\infty}\tfrac{Q_i(T)}{T}=0$.}%
\FinalVersion{\footnote{Please see Appendix~\ref{sec:extended} for a formal proof
and comments on how finite buffer capacity might affect the system.}}
So dividing by $T$ on both sides and letting $T\to\infty$ gives
\begin{align}
\bar{\kappa}_i = f_i - s_i,
\label{eq:dissatisfaction-equation-fi-less}
\end{align}
which concludes the proof for the $s_i<f_i$ case.

We now prove that $\bar{\kappa}_i=0$ if $s_i=f_i$.
For any policy $\{S_i(t)\}$ with an average service rate $s_i(=f_i)$,
define $\tilde{S}_i(t)=\frac{\tilde{s}}{s_i}S_i(t)$ for some $\tilde{s}<f_i$.
So \NewAfterReview{the long-term average}
of $\{\tilde{S}_i(t)\}$ is $\tilde{s}$.
Further, $\{\tilde{S}_i(t)\}$ is a valid policy since it satisfies all the policy constraints.
Using \eqref{eq:dissatisfaction-equation-fi-less}, we get
\begin{align*}
    \bar{\kappa}_i\left(\{\tilde{S}_i(t)\}\right) = f_i-\tilde{s}.
\end{align*}
Using the fact that $S_i(t)$ is strictly more than $\tilde{S}_i(t)$, and
the fact that $\tilde{s}$ can be any arbitrary value less than $f_i$, we get
\begin{align*}
    \bar{\kappa}_i(s_i=f_i) \le f_i-\tilde{s}\quad \forall \ \tilde{s}<f_i,
\end{align*}
which is equivalent to
\begin{align*}
    \bar{\kappa}_i \le 0 = f_i-s_i.
\end{align*}
Thus it holds that $\bar{\kappa}_i=f_i-s_i$ for all $s_i\le f_i$.

The argument for the $s_i=f_i$ case also applies
to any $s_i\ge f_i$.
Thus we have the more general expression
\begin{align*}
    \bar{\kappa}_i = \left(f_i-s_i\right)^+,
\end{align*}
which works for all $s_i\ge0$.
However, when $s_i>f_i$, the buffer \NewAfterReview{occupancy} $Q_i(\cdot)$ grows indefinitely large.
\end{proof}

\begin{proof}[Proof of Theorem~\ref{thm:lower-bound}]
Since any scheduling policy must satisfy $\sum_{i=1}^mS_i(t)\le c(t)$
for all times $t$, the long-term averages must satisfy $\sum_{i=1}^ms_i\le\bar{c}$.
So the constraint set includes all feasible policies.
For any feasible policy, Claim~\ref{claim:shortfall-equation} gives
the average shortfall at user $i$ as $\bar{\kappa}_i=\max(f_i-s_i,0)$.
Hence, when the mean consumptions  are known, the average dissatisfaction of any policy must be at least as large as
the minimum of $\frac{1}{m}\sum_{i=1}^mV_i(\max(f_i-s_i,0))$ subject to these constraints.
A similar argument applies when the mean consumptions are unknown.

\AC{Clearly, the simple policy satisfies the budget constraint at any time $t$: $\sum_{i=1}^m S_i(t)\le c(t)$. Also, the long-term average of this simple allocation is $\hat{s}_i$. Hence, by the assumption on $\hat{s}_i$ and Claim~\ref{claim:shortfall-equation}, the long-term average dissatisfaction is within $\Delta$ of the optimum.}
\end{proof}

\begin{remark}
\AC{While Theorem~\ref{thm:lower-bound} seems to give a simple policy,} it is not clear if there are efficient algorithms to solve \eqref{eq:conc-prog} and \eqref{eq:exp-conc-prog}.
This is because \eqref{eq:conc-prog} and \eqref{eq:exp-conc-prog} are non-convex problems, and 
standard gradient-based  methods \cite{BoydV2004,Neely2010} may get stuck at a local optimum.
The optimization problem in \eqref{eq:conc-prog} has been solved in \cite{BhimarajuZC2022} using a knapsack-type subroutine 
for the special case where $V_i(f_i)=V\cdot f_i$ for all $i$. \AC{However, the general case remains open, which we tackle here. Further, the problem in \eqref{eq:exp-conc-prog} is completely new and has an interesting and challenging non-convex structure, as discussed later.}
\label{rem:about-lower-bound}
\end{remark}

Note that the lower bound in Theorem~\ref{thm:lower-bound} in the known consumption case  holds even if the scheduling policy
knows $\{F_i(t)\}$ (and not just $\{f_i\}$).
This is because Claim~\ref{claim:shortfall-equation} makes no
assumptions on the inter-dependence between $F_i(t)$ and $S_i(t)$. 
\AC{However, the simple optimal policy 
in Theorem~\ref{thm:lower-bound} for this case only assumes knowledge of $\{f_i\}$.}

\section{LP Algorithm for Known $\{f_i\}$}
\label{sec:lp}

\AC{In this section, we present a polynomial-time algorithm for the non-convex optimization problem \eqref{eq:conc-prog} of Theorem~\ref{thm:lower-bound}. 
Towards that, we solve a} linearized 
program \eqref{prog:lp} to get the solutions $\{s_i^\textsc{lp}\}$.
Then, using Theorem~\ref{thm:lower-bound}'s simple policy, at each time $t$, we allocate $s_i^\textsc{lp}/\bar{c}$ fraction of
the available $c(t)$ units of resource to user $i$.
We state this as \AlgName{} (Alg.~\ref{alg:lp-allocation}) and its performance guarantee as Theorem~\ref{thm:perf-guarantee}.

\begin{align}
\underset{\{s_i\mathop{:}0\le s_i\le f_i\}}{\text{minimize}}&\ \  
\frac{1}{m}\sum_{i=1}^m \left(1-\frac{s_i}{f_i}\right)V_i(f_i) \tag{LinProg}\label{prog:lp}\\
\text{subject to}\ &\ \ \ \sum_{i=1}^m s_i \le \bar{c}. \nonumber
\end{align}

\begin{algorithm}
\algnewcommand\algorithmicforeach{\textbf{for each}}
\algdef{S}[FOR]{ForEach}[1]{\algorithmicforeach\ #1\ \algorithmicdo}
\caption{\AlgName}
\label{alg:lp-allocation}
\begin{algorithmic}[1]
\State Solve  \eqref{prog:lp} to get solutions $\{s_i^\textsc{lp}\}$
\For {time $t\in\{1,2,\ldots\}$}
    \ForEach {$i\in\{1,2,\ldots,m\}$}
        \State Allocate $S_i(t)=s_i^\textsc{lp}\frac{c(t)}{\bar{c}}$ units to user $i$
    \EndFor
\EndFor
\end{algorithmic}
\end{algorithm}

Note that \eqref{prog:lp} is equivalent to maximization of positively weighted sum of non-negative
variables subject to an upper bound on their sum (after ignoring constants and rescaling).
So we do not need general-purpose linear-programming solvers,
and it can be solved using a familiar greedy approach in $\mathcal{O}(m \log m)$ time.
We just need to assign $s_i=f_i$ in decreasing order of $\left\{\frac{V_i(f_i)}{f_i}\right\}$, 
and the final $i$ just before we run out of resource $\bar{c}$ gets the remaining resource.%
\footnote{This solution also satisfies the condition of Claim~\ref{claim:only-corners}.}

Note that we need to solve the linear program only once until $\{f_i\}$ and $\bar{c}$ change.
We then simply reuse the solution for each time instant to allocate
the resource.
In multimedia streaming, the set of users can change every few minutes.
Consumption and supply can also change with time in electricity markets and supply chain settings.
\NewAfterReview{Using an off-the-shelf solver to repeatedly
solve \eqref{eq:conc-prog} with new $\{f_i\}$ and $\bar{c}$ is intractable, and our linear program lets us do this efficiently.}

\begin{theorem}
\AlgName{} (Alg.~\ref{alg:lp-allocation}) always results in a feasible
allocation, and
the long-term average dissatisfaction achieved by \AlgName, $V^\AlgName$,
satisfies
\begin{align*}
V^\AlgName - \bar{V} \le \frac{C}{m}
\end{align*}
for some constant $C$,
where $\bar{V}$ is the lower bound from Theorem~\ref{thm:lower-bound}.
\label{thm:perf-guarantee}
\end{theorem}

Before we prove Theorem~\ref{thm:perf-guarantee}, we state Claim~\ref{claim:only-corners},
which characterizes the solution of \eqref{eq:conc-prog}, the non-convex
optimization program in Theorem~\ref{thm:lower-bound}.
This will be helpful in relating the solution of
\eqref{prog:lp} to the solution of \eqref{eq:conc-prog}.

\begin{claim}
\AC{There exists at least one global optimum $\{s_i^\textsc{cp}\}$ of \eqref{eq:conc-prog} which satisfies the following:}
\begin{align*}
|\{i: 0 < s_i^\textsc{cp} < f_i\}| \le 1,
\end{align*}
i.e., there is at most one user $i$ where the optimal solution
$s_i^\textsc{cp}$ is neither $0$ nor $f_i$. 
\label{claim:only-corners}
\end{claim}
\begin{proof}
Notice that \eqref{eq:conc-prog} can be written as 
the minimization of $\sum_{i=1}^m V_i(f_i-s_i)$ subject to $\sum_{i=1}^m s_i \le \bar{c}$ and $0\le s_i \le f_i$.
Since $V_i(\cdot)$ are all concave, the objective 
is also concave.
The constraint set is defined by linear constraints, and is thus a polyhedron.
Since the minima of a concave function over a polyhedron happens at a corner point,%
\footnote{See e.g., \url{https://math.stackexchange.com/q/2946023}.
}
the solution $\{s_i^\textsc{cp}\}$ is a corner point of the feasible set.

The vector $\{s_i^\textsc{cp}\}$ is $m$-dimensional, so corner
points are formed by the solution of $m$ linear equations.
Since $\sum_{i=1}^ms_i=\bar{c}$ is just one equation, at least $m-1$ of
the equations $\{s_i=0\}\cup\{s_i=f_i\}$ must be true.
(If $\sum_{i=1}^ms_i$ is strictly less than $\bar{c}$, then $m$ of
these equations must be true.)
Further, both $s_i=0$ and $s_i=f_i$ cannot be true simultaneously.
Thus at least $m-1$ of $\{s_i\}$ must be either $0$ or $f_i$, which implies
at most one $s_i$ is neither $0$ nor $f_i$.
\end{proof}

We are now ready to prove Theorem~\ref{thm:perf-guarantee}.

\begin{proof}[Proof of Theorem~\ref{thm:perf-guarantee}]
First observe that \eqref{prog:lp} is in the form of \eqref{eq:conc-prog} as well:
if $\tilde{V}_i(x)=x\frac{V_i(f_i)}{f_i}$, then
$\tilde{V}_i(f_i-s_i)=\left(1-\frac{s_i}{f_i}\right)V_i(f_i)$, which is 
the term in \eqref{prog:lp}.
Further, since $\tilde{V}_i(x)$ is linear, and thus concave, it satisfies all
the properties of $V_i(\cdot)$ assumed in Sec.~\ref{sec:assumptions} (note
that $\tilde{V}_i(0)=0$ and $\tilde{V}_i(\cdot)$ is increasing).
So Claim~\ref{claim:only-corners} applies to \eqref{prog:lp} as well,
and an optimal $\{s_i^\textsc{lp}\}$ satisfies
\begin{align*}
|\{i:0<s_i^\textsc{lp}<f_i\}| \le 1.
\end{align*}
While it is possible that a generic linear-programming solver
might give a different (global optimal) solution which does not
satisfy this condition, please see the discussion after Alg.~\ref{alg:lp-allocation} 
for an efficient and simple algorithm
which gives a solution to \eqref{prog:lp} satisfying Claim~\ref{claim:only-corners}.
Let $\mathcal{S}^\textsc{lp}$ denote  $\{i:s_i^\textsc{lp}<f_i\}$, the set of 
users who have a non-zero dissatisfaction, and let $\mathcal{S}^\textsc{cp}$
denote a similar set for $\{s_i^\textsc{cp}\}$.
Further, let $i^\textsc{lp}$  and $i^\textsc{cp}$ denote the users (if any)
who have $s_i>0$, but still have a non-zero dissatisfaction
(i.e., the one user identified in Claim~\ref{claim:only-corners}).

Since $\{s_i^\textsc{cp}\}$, the optimal solution of \eqref{eq:conc-prog},
also satisfies the constraints of \eqref{prog:lp}, $\{s_i^\textsc{cp}\}$
is a feasible solution for \eqref{prog:lp}.
But the optimal solution of \eqref{prog:lp} is given by $\{s_i^\textsc{lp}\}$.
So we have
\begin{align}
\tfrac{1}{m}\sum_{i\in\mathcal{S}^\textsc{lp}}\left(1-\tfrac{s_i^\textsc{lp}}{f_i}\right)V_i(f_i)
&\le \tfrac{1}{m}\sum_{i\in\mathcal{S}^\textsc{cp}}\left(1-\tfrac{s_i^{\textsc{cp}}}{f_i}\right)V_i(f_i).
\label{eq:lp-le-cp}
\end{align}
Claim~\ref{claim:only-corners} gives us
\begin{align*}
\tfrac{1}{m}\!\!\!\sum_{i\in\mathcal{S}^\textsc{lp}}\!\!\!\left(\!1\!\!-\!\!\tfrac{s_i^\textsc{lp}}{f_i}\!\right)\!\!V_i(f_i)
&= \tfrac{1}{m}\!\!\!\left(\sum_{i\in\mathcal{S}^\textsc{lp}\setminus\{i^\textsc{lp}\}}\!\!\!\!\!\!\!\!V_i(f_i)
\!+\! \left(\!1\!\!-\!\!\tfrac{s_{i^\textsc{lp}}^\textsc{lp}}{f_{i^\textsc{lp}}}\!\right)\!\!V_{i^\textsc{lp}}(f_{i^\textsc{lp}})\!\right)\!\!.
\end{align*}
However, \AC{by Claim~\ref{claim:shortfall-equation}, the long-term average cost satisfies} 
\begin{align*}
V^\textsc{LinAlloc} = \tfrac{1}{m}\sum_{i\in\mathcal{S}^\textsc{lp}\setminus\{i^\textsc{lp}\}}V_i(f_i)
+\tfrac{1}{m}V_{i^\textsc{lp}}(f_{i^\textsc{lp}}-s_{i^\textsc{lp}}^\textsc{lp}),
\end{align*}
which, together with the monotonicity of $V_{i^\textsc{lp}}(\cdot)$ gives us
\begin{align}
V^\textsc{LinAlloc} - \tfrac{1}{m}V_{i^\textsc{lp}}(f_{i^\textsc{lp}}) \le \tfrac{1}{m}\sum_{i\in\mathcal{S}^\textsc{lp}}\left(1-\tfrac{s_i^\textsc{lp}}{f_i}\right)V_i(f_i).
\label{eq:linalloc-le-lp}
\end{align}
Similar arguments on the right side of \eqref{eq:lp-le-cp} give us
\begin{align}
\frac{1}{m}\sum_{i\in\mathcal{S}^\textsc{cp}}\left(1-\frac{s_i^\textsc{cp}}{f_i}\right)V_i(f_i)
\le \bar{V} + \frac{1}{m}V_{i^\textsc{cp}}(f_{i^\textsc{cp}}).
\label{eq:cp-le-vbar}
\end{align}
Substituting \eqref{eq:linalloc-le-lp} and \eqref{eq:cp-le-vbar} into
\eqref{eq:lp-le-cp} gives us
\begin{align*}
V^\textsc{LinAlloc} \le \bar{V} + \tfrac{1}{m}\left(V_{i^\textsc{lp}}(f_{i^\textsc{lp}})+V_{i^\textsc{cp}}(f_{i^\textsc{cp}})\right).
\end{align*}
From the assumptions on \NewAfterReview{$\{F_i(t)\}$ in Sec.~\ref{sec:assumptions}}, $\{f_i\}$ are bounded almost surely, and there is a $C$ such that
\begin{align*}
{\scriptstyle V^\textsc{LinAlloc} - \bar{V} \le \tfrac{C}{m},}
\end{align*}
which concludes the proof of Theorem~\ref{thm:perf-guarantee}.
\end{proof}

\FinalVersion{Note that \eqref{prog:lp} effectively solves the problem with a gap of at most one user.
Since the cost we have considered is the \emph{average} across the $m$ users, the simple policy of Theorem~\ref{thm:lower-bound} with the solution of \eqref{prog:lp} is within
$\mathcal{O}\!\!\left(\frac{1}{m}\right)$ of the optimum.
}

\section{Unknown Consumption Statistics}
\label{sec:unknown}

\AC{In this section, we consider the objective in \eqref{eq:EXP-avg-dissatisfaction}.
Here, the ergodic consumption
rates $\{f_i\}$ are not known, but we know the distributions $\{p_i\}$ of $\{f_i\}$.
We assume $\{p_i\}$ have finite supports $\{[a_i,b_i]\}$. 
}
\AC{As in the known consumption case, 
we solve the static non-convex optimization problem in \eqref{eq:exp-conc-prog} and then employ the simple policy in Theorem~\ref{thm:lower-bound}.
Defining $\mathcal{K}_i(s_i)=\EX_{f_i\sim p_i}\left[V_i\left((f_i-s_i)^+\right)\right]$, 
\eqref{eq:exp-conc-prog} becomes}
\begin{align}
    \underset{\{s_i\mathop{:}0\le s_i\le b_i\}}{\text{minimize}}\ & \ \ \frac{1}{m}\sum_{i=1}^m \mathcal{K}_i(s_i)
    \ \ \text{s.t.} \ \ \sum_{i=1}^m s_i \le \bar{c}.
\label{eq:unknown-symmetric}
\end{align}
\AC{While an intractable problem in general,
we consider the case where $\{p_i\}$ are {\em non-increasing} over their respective supports.
Many well known densities  satisfy this including uniform, truncated exponential, and Gaussian restricted above mean.
}

\AC{
\begin{claim}
When $p_i$ is non-increasing, $\mathcal{K}_i(\cdot)$ is a
non-increasing function that is concave over $[0,a_i]$ and convex over $[a_i,b_i]$. 
\label{lem:pi-nonincreasing}
\end{claim}
\begin{proof}
Please see Appendix~\ref{sec:conv-conc}.
\end{proof}
}

\AC{Thus, when $\{p_i\}$ are non-increasing, the objective becomes minimization of sum of non-increasing functions subject to linear constraint, where each function is part concave and part convex.
Though this structure may be of independent mathematical interest,  solving it in its full generality seems quite challenging, and we leave it for future work.
We consider the symmetric case where $a_i=a$, $b_i=b$, and $\mathcal{K}_i=\mathcal{K}$
for all $i$. This happens when $\{V_i\}$ are the same and the consumption rates 
are all drawn from the same distribution. }
\AC{This is a natural assumption for multimedia streaming and energy grid where all the users are treated equal.
}

We first state and prove Claim~\ref{claim:convex-equal} and Claim~\ref{claim:concave-corners} which
characterize the optimal solution of \eqref{eq:unknown-symmetric}.

\begin{claim}
There exists a (global) optimum $\{s_i^*\}$ of \eqref{eq:unknown-symmetric}, with
\FinalVersion{$\mathcal{S}^*_\text{conv}=\{i:a\le s_i^*\le b\}$},
such that $s_i^*=s_j^*$ for $i,j\in\mathcal{S}^*_\text{conv}$.
\label{claim:convex-equal}
\end{claim}
\begin{proof}
Suppose $s_i^*>s_j^*\ge a$ for some $i,j\in\mathcal{S}^*_\text{conv}$.
Since $\mathcal{K}(\cdot)$ is convex in $[a,b]$, Jensen's inequality
gives
\begin{align*}
    {\scriptstyle2\mathcal{K}\left(\tfrac{s_i^*+s_j^*}{2}\right) \le \mathcal{K}(s_i^*) + \mathcal{K}(s_j^*).}
\end{align*}
So we can reduce the cost by setting both users to $\frac{s_i^*+s_j^*}{2}$.
(Note that the constraint on cumulative service rate being less than $\bar{c}$
is not violated by this transformation.)
So we can always find an optimal $\{s_i^*\}$ with $s_i^*=s_j^*$
for $i,j\in\mathcal{S}^*_\text{conv}$.
\end{proof}

\begin{claim}
There exists a (global) optimum of \eqref{eq:unknown-symmetric} $\{s_i^*\}$, with
\FinalVersion{$\mathcal{S}^*_\text{conc}=\{i:0\le s_i^*< a\}$},
such that \FinalVersion{$|\{i \in \mathcal{S}^*_\text{conc}:s_i^*\neq0\}|\le1$}, and $\{s_i^*\}$ satisfies Claim~\ref{claim:convex-equal}.
\label{claim:concave-corners}
\end{claim}
\begin{proof}
This is similar to Claim~\ref{claim:only-corners}, with the
functions here being $\mathcal{K}(\cdot)$ rather than $V_i(\cdot)$.
In the set $\mathcal{S}^*_\text{conc}$, the search for the optimal
solution can be restricted to \FinalVersion{$s_i\in[0,a)$} (by definition of $\mathcal{S^*_\text{conc}}$),
and in this region, $\mathcal{K}(\cdot)$ is concave.
Hence, for these users, similar arguments as Claim~\ref{claim:only-corners} apply,
and give \NewAfterReview{the cardinality bound}.
\FinalVersion{One difference with Claim~\ref{claim:only-corners} is that the interval is open at $a$,
avoiding the $s_i=a$ case we had there.
This is because if $s_i=a$ for any $i$, it is a part of the convex interval and the transformation of 
Claim~\ref{claim:convex-equal} applies.%
\footnote{We thank the anonymous reviewer who pointed this out.}}
None of \FinalVersion{the users in $\mathcal{S}^*_\text{conc}$} are affected by Claim~\ref{claim:convex-equal}'s transformation,
so the overall solution also satisfies Claim~\ref{claim:convex-equal}.
\end{proof}

\begin{algorithm}
\caption{\NoFbAlgName}
\label{alg:symmetric}
\begin{algorithmic}[1]
\State Initialize $V^*\gets\infty$, $s_i^*\gets0\ \forall \ i\in\{1,2,\ldots,m\}$
\For {\FinalVersion{$n\in\{0,1,\ldots,m-1\}$}}
        \State $V,\beta\gets$ solution of \eqref{prog:find-vstar}
        \If {$V<V^*$}
            \State $V^*\gets V$, $\beta^*\gets\beta$, \FinalVersion{$n^*\gets n$}
        \EndIf
\EndFor
\State $s_1^*\gets\beta^*$
\For {\FinalVersion{$i=2$ \textbf{to} $n^*+1$}}
    \State \FinalVersion{$s_i^*\gets \frac{\bar{c}-\beta^*}{n^*}$}
\EndFor
\State At each time $t$, allocate using $S_i(t)=s_i^*\frac{c(t)}{\bar{c}}$
\end{algorithmic}
\end{algorithm}

Claim~\ref{claim:convex-equal} and Claim~\ref{claim:concave-corners} imply that there is
\NewAfterReview{at least one} optimal solution that has
\FinalVersion{one user with service rate $\beta\in[0,a)$, 
and $n$ users with service rate $r\in[a,b]$ with $r=\frac{\bar{c}-\beta}{n}$.
The remaining $m-n-1$ users have a service rate $0$.
If we have the value of $n$, we can find the optimal
$\beta$ by solving \eqref{prog:find-vstar}:
\begin{align}
&V^*\!=\min_\beta\! {\scriptstyle \left(\!\!\mathcal{K}\!(\beta)
+n\mathcal{K}\!\left(\!\tfrac{\bar{c}-\beta}{n}\!\right)
\!+\! (m\!-\!n\!-\!1)\mathcal{K}\!(0)\!\right)\!},
\label{prog:find-vstar} \\
&\text{s.t.}\ \max(0,\bar{c}\!-\!nb) \le \beta \le \min(a, \bar{c}\!-\!na), \nonumber
\end{align}
to get the optimal $V^*$ and $\beta^*$ (given $n$).
If some $n$ results in an empty feasible region
for $\beta$,
we set $V^*=\infty$.}
Since \eqref{prog:find-vstar} is an optimization in one variable
over a bounded range, it can be solved using \NewAfterReview{brute force}.
These observations directly lead us to Alg.~\ref{alg:symmetric} and  Theorem~\ref{thm:symmetric}.%
\footnote{While \AlgName{} has an $\mathcal{O}\left(\frac{1}{m}\right)$ gap,
\NoFbAlgName{} achieves optimal cost.
This is because the ``symmetric'' assumption helps us solve \eqref{eq:unknown-symmetric}
exactly.
If we make the same assumption in the \AlgName{} setting, we can solve the optimization program
faster than \NoFbAlgName{} since we know $\beta=\bar{c}-na$  as $a=b$.
Note that the time complexity of \NoFbAlgName{} is $\mathcal{O}\left(mT\right)$,
where $T$ is the complexity of \eqref{prog:find-vstar}.}

\begin{theorem}
The allocation policy in \NoFbAlgName{} (Alg.~\ref{alg:symmetric}) achieves 
the minimum defined by \eqref{eq:unknown-symmetric}.
\label{thm:symmetric}
\end{theorem}

\section{Conclusion \& Future Work}
\label{sec:conclusion}

In this paper, we have considered a resource allocation problem to minimize
user dissatisfaction due to shortfalls in the allocated (time-varying) resource.
This framework is applicable in varied scenarios such as wireless multimedia streaming,
renewable energy grid, and supply chain logistics.
We first reduced the resource allocation problem to a static non-convex optimization problem,
and exploited the combinatorial structure of the problem to develop algorithms that are either near-optimal (when the users' consumption
statistics are known) or optimal (when the consumption statistics are unknown, but the users are symmetric).

This still leaves us with many interesting open questions.
We highlight \NewAfterReview{four} here.
(i)~Can we solve \eqref{eq:conc-prog} with a smaller gap than
$\mathcal{O}\left(\frac{1}{m}\right)$ in polynomial time?
While \cite{BhimarajuZC2022} tackles this under
some restrictive assumptions (see Remark~\ref{rem:diff-tcom}),
it is still open if it can be solved in the general setting considered in this
paper.
(ii)~While this paper minimizes the cumulative dissatisfaction, it is
also important to consider fairness in the resource allocation \cite{KodyWM2022}.
Solving the problems considered in this paper in conjunction with fairness
considerations is an important future direction.
(iii)~Can we solve \eqref{eq:exp-conc-prog} when the users are not symmetric?
In particular, it would be interesting to see if \eqref{eq:exp-conc-prog} can be solved for
asymmetric users when $\{V_i(\cdot)\}$ have more structure.
\NewAfterReview{(iv) Can we derive similar results when the buffer capacity is finite?
Specifically, does knowledge of $\{Q_i(t)\}$ help with the scheduling in this case?}
\FinalVersion{We briefly comment on how a finite buffer affects Claim~\ref{claim:shortfall-equation}
in Appendix~\ref{sec:extended}.
Note that we did not make use of the infinite
buffer assumption beyond Claim~1 in this paper.}

\appendices

\section{Proof of Claim~\ref{lem:pi-nonincreasing}}
\label{sec:conv-conc}

For $s\in[0,b_i]$, the function we are interested in is
\begin{align*}
    {\scriptstyle \mathcal{K}_i(s)
    = \EX_{f_i\sim p_i}\left[V_i\left((f_i-s)^+\right)\right]
    = \int_{a_i}^{b_i} V_i\left((f_i-s)^+\right) p_i(f_i)df_i.}
\end{align*}
Using $V_i(0)=0$, this can be simplified to
\begin{align*}
    {\scriptstyle \mathcal{K}_i(s)=
    \int_{\max\{s,a_i\}}^{b_i} V_i(f_i-s)p_i(f_i)df_i.} 
\end{align*}

Using the Leibniz rule of differentiation (see \cite{ProtterM2012}, for example)
in the $s\in(0,a_i)$ region, we get
\begin{align*}
    {\scriptstyle\mathcal{K}_i''(s) = \int_{a_i}^{b_i} V_i''(f_i-s)p_i(f_i)df_i.}
\end{align*}
However, $V_i(\cdot)$ is concave, which implies $V_i''(\cdot)\le0$,
giving
\begin{align}
    {\scriptstyle\mathcal{K}_i''(s) \le 0 \quad \text{for} \ s\in(0,a_i).}
    \label{eq:conc-le-a}
\end{align}

In the $s\in(a_i,b_i)$ region, we get (using $V_i(0)=0$)
\begin{align*}
{\scriptstyle
    \mathcal{K}_i'(s) 
    = -\int_s^{b_i} V_i'(f_i-s)p_i(f_i)df_i.}
\end{align*}
Using the Leibniz rule again gives us
\begin{align*}
    {\scriptstyle\mathcal{K}_i''(s) = V_i'(0)p_i(s)+\int_s^{b_i}V_i''(f_i-s)p_i(f_i)df_i.}
\end{align*}

Using the non-increasing property of $p_i$, 
we have
$\int_s^{b_i}V_i''(f_i-s)p_i(f_i)df_i\ge p(s)\int_s^{b_i}V_i''(f_i-s_i)df_i$, which gives
\begin{align}
    {\scriptstyle\mathcal{K}_i''(s) \ge 0 \quad \text{for} \ s \in (a_i,b_i).}
    \label{eq:conv-ge-a}
\end{align}
Using \eqref{eq:conc-le-a} and \eqref{eq:conv-ge-a} together gives us
the required result.

\section{Proof of $s<f\implies\lim_{T\to\infty}\frac{Q(T)}{T}=0$}
\label{sec:extended}

While proving Claim~\ref{claim:shortfall-equation}, we assumed that when $s_i<f_i$,
we have $\lim_{T\to\infty}\frac{Q_i(T)}{T}=0$.
In Theorem~\ref{thm:stable-queue} below, we prove this statement formally.

To avoid clutter, we drop the subscript  $i$ and define the system as follows.
We have a queue $Q(t)$ that is being served by a process $S(t)$
and utilized by a consumption process $F(t)$.
We have the following system dynamics:
\begin{align*}
Q(t+1) = \left(Q(t)+S(t)-F(t)\right)^+.
\end{align*}
The following theorem holds.
\begin{theorem}
Assume the long-term averages for $S(\cdot)$ and $F(\cdot)$ exist:
\begin{align*}
\lim_{T\to\infty}\frac{1}{T}\sum_{t=1}^T S(t) = s\quad\text{and}\quad \lim_{T\to\infty}\frac{1}{T}\sum_{t=1}^T F(t) = f.
\end{align*}
If $s<f$, then
\begin{align*}
\lim_{T\to\infty}\frac{Q(T)}{T} = 0.
\end{align*}
\label{thm:stable-queue}
\end{theorem}
\begin{proof}
Let us first show that $\{Q(t)\}$ is $0$ infinitely often if $s<f$.
For this, assume the contrary, i.e., there exists an $N_Q\in\mathbb{N}$ such that $Q(t)>0$
for all $t\ge N_Q$.
This means that the system dynamics reduces to 
\begin{align}
Q(t+1)=Q(t)+S(t)-F(t)
\label{eq:no-zero}
\end{align}
for all $t\ge N_Q$.
For any $T>N_Q$, adding up the equations for $t=N_Q,N_Q+1,\ldots,T-1$, we get
\begin{align*}
Q(T)=Q(N_Q)+\sum_{t=N_Q}^{T-1}S(t) - \sum_{t=N_Q}^{T-1}F(t).
\end{align*}
However, since $s<f$, for a large enough $T$, the sum of $\{S(t)\}$ would be
less than the sum of $\{F(t)\}$ (by a large enough margin), giving us $Q(T)<0$,
which is a contradiction.
Thus $\{Q(t)\}$ is $0$ infinitely often.

To prove the theorem, we need to prove that for any $\epsilon>0$, we can find
an $N_\epsilon\in\mathbb{N}$ such that $\left|\frac{Q(T)}{T}\right|<\epsilon$
for all $T\ge N_\epsilon$.
Since the long-term averages for $\{S(t)\}$ and $\{F(t)\}$ exist, we can find
an $\tilde{N}_\epsilon$ such that
\begin{align}
\left|s-\frac{1}{T-1}\sum_{t=1}^{T-1}S(t)\right| < \frac{\epsilon}{4}\quad\text{and}\quad
\left|f-\frac{1}{T-1}\sum_{t=1}^{T-1}F(t)\right| < \frac{\epsilon}{4},
\label{eq:limit-sf}
\end{align}
for all $T\ge \tilde{N}_\epsilon$.
Let us define $N_\epsilon$ as follows:
\begin{align*}
N_\epsilon := \min\{t:Q(t)=0\ \text{and}\ t\ge\tilde{N}_\epsilon\}.
\end{align*}
Since $\{Q(t)\}$ is $0$ infinitely often, $N_\epsilon$ exists (i.e., the minimum is over
a non-empty set).
We show that this $N_\epsilon$ satisfies the limit condition: $\left|\frac{Q(T)}{T}\right|<\epsilon$
for all $T\ge N_\epsilon$.
For any $T\ge N_\epsilon$, define $T'$ as follows:
\begin{align*}
T' := \max\{t:Q(t)=0\ \text{and}\ t\le T\}.
\end{align*}
So $T'$ is the last time $Q(t)$ is $0$ before $t=T$.
Since $T\ge N_\epsilon$ and $Q(N_\epsilon)=0$ (by definition), we have $T'\ge N_\epsilon$.

Using \eqref{eq:limit-sf}, we get
\begin{align*}
(T-1)\left(s-\frac{\epsilon}{4}\right)<\sum_{t=1}^{T-1}S(t)<(T-1)\left(s+\frac{\epsilon}{4}\right)
\end{align*}
and
\begin{align*}
(T-1)\left(f-\frac{\epsilon}{4}\right)<\sum_{t=1}^{T-1}F(t)<(T-1)\left(f+\frac{\epsilon}{4}\right)
\end{align*}
for all $T\ge T'$ (since $T'\ge N_\epsilon\ge\tilde{N}_\epsilon$ by definition).
Since these inequalities are also true for $T=T'$, we can subtract the summation
till $t=T'-1$ to get
\begin{align*}
\sum_{t=T'}^{T-1}S(t) < (T-T')s + (T+T')\frac{\epsilon}{4}
\end{align*}
and
\begin{align*}
\sum_{t=T'}^{T-1}F(t) > (T-T')f - (T+T')\frac{\epsilon}{4}.
\end{align*}
Observe that from the definition of $T'$, we have $Q(t)>0$ for $T'<t<T$.
So the system dynamics again reduce to \eqref{eq:no-zero} in this range
and a telescopic summation gives
\begin{align*}
Q(T) < Q(T') + (T-T')(s-f) + 2(T+T')\frac{\epsilon}{4}.
\end{align*}
Using $Q(T')=0$ and $T'\le T$, we get for all $T\ge N_\epsilon$,
\begin{align*}
\frac{Q(T)}{T} < \frac{T-T'}{T}(s-f) + \epsilon,
\end{align*}
and since $s<f$ and $Q(T)\ge0$, this gives us
\begin{align*}
\left|\frac{Q(T)}{T}\right| < \epsilon
\end{align*}
for all $T\ge N_\epsilon$, which concludes the proof.
\end{proof}

One of the assumptions we made in order to arrive at Claim~\ref{claim:shortfall-equation}
is that the system has infinite buffer capacity, i.e., $Q(\cdot)$ can grow without bound.
While this is indeed a reasonable assumption in some applications, it may not be valid 
in general.
If our buffer only has finite capacity $B$, the dynamics would change as
\begin{align*}
Q(t+1) = \min\{\left(Q(t)+S(t)-F(t)\right)^+, B\}.
\end{align*}
Retracing the steps in the proof of Claim~\ref{claim:shortfall-equation},
we instead end up with
\begin{align*}
\kappa(t)\ge Q(t+1) - (Q(t)+S(t)-F(t)),
\end{align*}
giving us
\begin{align*}
\bar{\kappa} \ge (f-s)^+.
\end{align*}
In this case, even if $s>f$, i.e, the service is greater than the consumption,
it is still possible to have non-zero average shortfall.
Thus it becomes more important to consider the buffer occupancy while scheduling
so as to not provide wasteful service to an already full buffer.

\bibliographystyle{IEEEtran}
\bibliography{refs}

\end{document}